\documentclass[oneside,11pt,reqno]{amsart}
\usepackage{amssymb,amsmath,amsthm,bbm,enumerate,mdwlist,url,multirow,hyperref,amsthm,stmaryrd}
\usepackage[pdftex]{graphicx}
\usepackage[shortlabels]{enumitem}

\theoremstyle{definition}
\newtheorem{definition}{Definition}
\numberwithin{definition}{section}
\theoremstyle{theorem}
\newtheorem{proposition}[definition]{Proposition}

\newtheorem{theorem}[definition]{Theorem}
\newtheorem{corollary}[definition]{Corollary}
\newtheorem{assumption}[definition]{Assumption}
\numberwithin{equation}{section}
\theoremstyle{remark}
\newtheorem{remark}[definition]{Remark}
\newtheorem{question}[definition]{Question}

\def\PP{\mathsf P}
\def\PPP{\mathcal P}
\def\expp{\mathsf{e}}
\def\Exp{\mathrm{Exp}}
\def\EE{\mathsf E}
\def\drift{\mathsf d}
\def\dd{\mathsf D}
\def\mm{\mathsf M}
\def\jj{\mathsf J}
\def\gg{\mathsf{g}}
\def\aa{\mathsf{a}^{p,\alpha}}
\def\aaa{\mathsf{a}^{p,\alpha'}}
\def\bb{\mathsf{b}^{p,\alpha}}
\def\cc{\mathsf{c}^{p,\alpha}}
\def\bbb{\mathsf{b}^{p,\alpha'}}
\def\QQ{\mathsf Q}

\def\II{\mathcal{I}^{p,\alpha}}
\def\III{\mathcal{I}^{p,\alpha'}}
\def\JJ{\mathcal{J}^{p,\alpha}}

\def\jump{J}
\def\LL{\mathcal{L}}
\def\HH{\mathcal{H}}
\def\KK{\mathcal{K}^{p,\alpha}}
\def\MM{\mathcal{M}^{p,\alpha}_\beta}
\def\NN{\mathcal{N}^{p,\alpha}_\beta}
\def\MMM{\mathcal{M}^{p,\alpha'}_\beta}
\def\NNN{\mathcal{N}^{p,\alpha'}_\beta}
\def\DD{\mathbb D}
\def\DDD{\mathcal D}
\def\FF{\mathcal F}
\def\nn{\mathsf{n}}
\def\diffusion{\sigma^2}
\def\kk{\mathsf{k}}
\def\GG{\mathcal G}

\def\Wm{W^{(q)}}
\def\AA{\mathcal A}
\def\BB{\mathcal B}

\begin{document}
\title[A  temporal  factorization at the maximum for certain pssMp]{A temporal factorization at the maximum for spectrally negative positive self-similar Markov processes}

\author{Matija Vidmar}
\address{Department of Mathematics, University of Ljubljana, Slovenia}
\email{matija.vidmar@fmf.uni-lj.si}

\begin{abstract}
For a spectrally negative positive self-similar Markov process with an a.s. finite overall supremum we provide, in tractable detail, a kind of conditional Wiener-Hopf factorization at the maximum of the absorption time at zero, the conditioning being on the overall supremum and the jump at the overall supremum. In a companion result the Laplace transform of said absorption time (on the event that the process does not go above a given level) is identified under no other assumptions (such as the process admitting a recurrent extension and/or hitting zero continuously), generalizing some existent results in the literature. 
\end{abstract}

\thanks{Financial support from the Slovenian Research Agency is acknowledged (research core funding No. P1-0222).}

\keywords{Spectrally negative L\'evy processes; positive self-similar Markov processes; splitting at the maximum; Wiener-Hopf factorization; martingales; time-changes}

\subjclass[2010]{Primary: 60G51, 60G18; Secondary: 60G44} 

\maketitle

\section{Introduction}

A fundamental feature of real-valued L\'evy processes is the independence of the pre-supremum process and of the post-supremum increments of the process, before an independent exponential random time, together with the associated spatio-temporal Wiener-Hopf factorization at the maximum. Through the Lamperti transform for positive self-similar Markov processes (pssMp) this splitting at the maximum of the underlying L\'evy process carries over, in particular, to a conditional, given the value of the overall maximum and the multiplicative jump at the maximum, independence of the time at which the ultimate supremum of the associated pssMp is reached and the time from then until its absorption at zero. Moreover, in the spectrally negative case the, roughly speaking, ``temporal conditional Wiener-Hopf factors'' corresponding to this independence statement can be made explicit. We proceed now to look at this in precise detail.

Let indeed $X=(X_t)_{t\in [0,\infty)}$ be a spectrally negative L\'evy process (snLp) under the probabilities $(\PP_x)_{x\in \mathbb{R}}$ in the filtration $\FF=(\FF_t)_{t\in [0,\infty)}$. This means that $X$ is a  c\`adl\`ag, real-valued $\FF$-adapted process with stationary independent\footnote{Without further qualification it means  ``stationary independent under $\PP_x$ for all $x\in \mathbb{R}$'', similarly when  ``a.s.'' appears with no further qualification it means ``a.s.-$\PP_x$ for all $x\in \mathbb{R}$'', etc.}  increments relative to $\FF$, no positive jumps and non-monotone paths, which, under $\PP_0$,  a.s. vanishes at zero; furthermore, for each $x \in \mathbb{R}$, the law of $X$ under $\PP_x$ is that of $x+X$ under $\PP_0$. We refer to \cite{bertoin,kyprianou,sato,doney} for the general background on (the fluctuation theory of)  L\'evy processes and to \cite[Chapter~VII]{bertoin} \cite[Chapter~8]{kyprianou} \cite[Chapter~9]{doney}  \cite[Section~9.46]{sato} for snLp in particular. As usual we set $\PP:=\PP_0$. Let also $\expp$ be an a.s. strictly positive $\FF$-stopping time such that for some (then unique) $p\in [0,\infty)$, $\PP_x[g(X_{t+s}-X_t)\mathbbm{1}_{\{\expp>t+s\}}\vert\FF_t]=\PP[g(X_s)]e^{-p s}\mathbbm{1}_{\{t<\expp\}}$ a.s.-$\PP_x$ for all $x\in \mathbb{R}$, whenever $\{s,t\}\subset [0,\infty)$ and $g\in \mathcal{B}_\mathbb{R}/\mathcal{B}_{[0,\infty]}$;\footnote{Throughout we will write $\QQ[W]$ for $\EE_\QQ[W]$, $\QQ[W;A]$ for $\EE_\QQ[W\mathbbm{1}_A]$ and $\QQ[W\vert \mathcal{H}]$ for $\EE_\QQ[W\vert \mathcal{H}]$. More generally the integral $\int fd\mu$ will be written $\mu(f)$ etc. For $\sigma$-fields $\AA$ and $\BB$, $\AA/\BB$ will denote the set of $\AA/\BB$-measurable maps; $\mathcal{B}_A$ is the Borel (under the standard topology) $\sigma$-field on $A$.}  in particular $\expp$ is exponentially distributed with rate $p$ ($\expp=\infty$ a.s. when $p=0$) independent of $X$. Finally let $\alpha\in (0,\infty)$. 

\begin{remark}
The conditions on the filtration $\FF$ are natural: If $X$ is adapted, and has independent increments relative to some filtration $\HH=(\HH_t)_{t\in [0,\infty)}$ with $\expp$ independent of $\HH_\infty$ (in particular if $\HH$ is the (completed) natural filtration of $X$ with $\expp$ independent of $X$), then all our assumptions are satisfied if, ceteris paribus, $\FF$ is not given a priori, but rather the filtration $(\HH_t\lor \sigma(\{\{u<\expp\}:u\in [0,t]\}))_{t\in [0,\infty)}$ (viz. the progressive enlargement of $\HH$ by $\expp$, i.e. the smallest enlargement of $\HH$ making $\expp$ into a stopping time) features in lieu of it.
\end{remark}

Now, associated to $X$, $\expp$ and $\alpha$, via the Lamperti transformation \cite{lamperti} (see also \cite[Theorem 13.1]{kyprianou}), is a positive $\alpha^{-1}$-self-similar Markov process $Y=(Y_s)_{s\in [0,\infty)}$, where we understand ``positive'' to mean that $0$ is an absorbing state. We make this precise:

Let $\DD$ be the space of real-valued c\`adl\`ag paths on $[0,\infty)$, endowed with the sigma-field $\DDD$ and canonical filtration $(\DDD_t)_{t\in [0,\infty)}$ of evaluation maps, shift operators $(\theta_t)_{t\in [0,\infty)}$, coordinate process $\xi=(\xi_t)_{t\in[0,\infty)}$. Set $$\mathrm{I}_t:=\int_0^{t }e^{\alpha \xi_u}du,\quad t\in [0,\infty];$$
$$\varphi_s:=\inf\{t\in [0,\infty):\mathrm{I}_t>s\},\quad s\in [0,\infty);$$
and for further an $l\in [0,\infty]$, $$\LL^{l}_s:=
\begin{cases}
e^{ \xi_{\varphi_s}} & \text{for }s\in [0,\mathrm{I}_l)\\
0 & \text{for }s\in [\mathrm{I}_l,\infty).
\end{cases}
$$ Then $Y_s=\LL^{\expp}_s(X)$ for $s\in [0,\infty)$. 

We will write $I_t:=\mathrm{I}_t(X)$, $t\in [0,\infty]$, and $\phi_s:=\varphi_s(X)$, $s\in [0,\infty)$, for short. Define also the filtration $\GG=(\GG_s)_{s\in [0,\infty)}$ by $\GG_s:=\FF_{\phi_s}$ for $s\in [0,\infty)$; $T_0:=\inf\{t\in (0,\infty):Y_t=0\}=I_\expp=\int_0^\expp e^{\alpha X_u}du$; 
and set for convenience $\QQ_y:=\PP_{\log y}$ for $y\in (0,\infty)$ (naturally $\QQ:=\QQ_1$). 

The assumptions on  $X$, $\expp$ and $\FF$ entail that for any $\FF$-stopping time $S$, on $\{S<\expp\}$, $\FF_S$ is independent of $((X_{S+u}-X_S)_{u\in [0,\infty)},\expp-S)$, which has (assuming the probability of $\{S<\expp\}$ is strictly positive) the distribution of $(X,\expp)$ under $\PP_0$ (one proves it first for deterministic $S$, then for $S$ assuming countably many values, then passes to the limit by approximating $S$ from above, in the usual manner). In consequence $Y$ is Markov with life-time $T_0$, cemetery state $0$, in the  filtration $\GG$, under the probabilities $(\QQ_y)_{y\in (0,\infty)}$: clearly it is $\GG$-adapted; moreover, for any $h\in \mathcal{B}_\mathbb{R}/\mathcal{B}_{[0,\infty]}$, $y\in (0,\infty)$, $\{s_1,s_2\}\subset [0,\infty)$, one has a.s.-$\QQ_y$, $\QQ_y[h(Y_{s_1+s_2})\mathbbm{1}_{\{s_1+s_2<T_0\}}\vert \GG_{s_1}]=\QQ_{Y_{s_1}}[h(Y_{s_2});s_2<T_0]\mathbbm{1}_{\{s_1<T_0\}}$ (and then, for any $s\in [0,\infty)$ and $H\in \DDD/\mathcal{B}_{[0,\infty]}$,  $\QQ_y[H(\theta_s(Y))\vert \GG_s]=\QQ_{Y_s}[H(Y)]$ a.s.-$\QQ_y$ on $\{s<T_0\}$). Besides, $Y$ respects the $1/\alpha$-self similarity property: for each $c\in (0,\infty)$ and $y\in (0,\infty)$, the law of $(cY_{sc^{-\alpha}})_{s\in [0,\infty)}$ under $\QQ_y$ is that of $Y$ under $\QQ_{cy}$. 

Conversely, any sufficiently regular positive $1/ \alpha$-self-similar Markov process with no positive jumps and non-monotone paths is got from some (possibly killed) snLp $X$ by the transformation $\LL$  \cite[Theorem 13.1]{kyprianou}. 

We refer to \cite[Chapter 13]{kyprianou} for a further account of the properties of pssMp, to  \cite{ckr} for their general fluctuation theory, and to \cite[Chapter 13.7]{kyprianou} for those/that of the spectrally negative type in particular.


Denote next by $\overline{Y}=(\overline{Y}_s)_{s\in [0,\infty]}$ (resp. $\overline{X}=(\overline{X}_t)_{t\in [0,\infty]}$) the running supremum process of $Y$ (resp. $X$) and define  $$L:=\sup\{s\in [0,\infty):\overline{Y}_s=Y_s\}\text{ and }G:=\sup\{t\in [0,\expp):\overline{X}_t=X_t\}.$$ 

We insist that 
\begin{assumption}\label{assumption}
When $p=0$, then $X$ drifts to $-\infty$.
\end{assumption}
This assumption is equivalent to the following being  a.s. finite quantities:  $\overline{Y}_\infty=\overline{Y}_L=e^{\overline{X}_G}=e^{\overline{X}_\expp}$ -- the overall supremum of $Y$; $L=I_G=\int_0^Ge^{\alpha X_u}du$ -- the last time $Y$ is at its running supremum; and $T_0$ -- the absorption time of $Y$. 
Indeed the following trichotomy is well-known: a.s., $Y$ never reaches zero and its overall supremum is infinite, hits zero continuously, or hits zero by a jump, according as $X$ does not drift to $-\infty$ and $p=0$, $X$ drifts to $-\infty$ and $p=0$, or $p>0$.

By the independence statement of the Wiener-Hopf factorization for $X$ (see Subsection~\ref{wiener-hopf} below) we have that (I) the pair $(L,\overline{Y}_\infty)$ is independent of   $\jump:=\frac{Y_L}{\overline{Y}_\infty}\mathbbm{1}_{\{L<T_0\}}+\mathbbm{1}_{\{L=T_0\}}\overset{\text{a.s.}}{=}e^{X_G-\overline{X}_\expp}$, which is the size of the multiplicative jump at $L$ for $Y$ on $\{L<T_0\}$ and $1$ otherwise: $\jump$ is only not a.s. equal to $1$, when $0$ is irregular for $(-\infty,0)$  for the process $X$  (which is equivalent to $X$ being of finite variation); and indeed (II) conditionally on $\overline{Y}_\infty$ and $\jump$ (or just $\overline{Y}_\infty$), $L$ is independent of 
$T_0-L=\int_G^\expp e^{\alpha X_u}du=e^{\alpha \overline{X}_G}e^{\alpha(X_G-\overline{X}_G)}\int_G^\expp e^{\alpha (X_u-X_G)}du$, which is the amount of time that elapses from $Y$ reaching its overall supremum and until absorption at zero. 

As indicated briefly at the start, this may be interpreted as a kind of conditional Wiener-Hopf factorization at the maximum of $T_0$, and in this paper we provide explicitly the associated ``conditional Wiener-Hopf factors''. That is to say, we compute, in tractable detail, for $y\in (0,\infty)$, $\beta\in [0,\infty)$,   the conditional Laplace transforms (i)  $\QQ_y\left[e^{-\beta L}\vert \overline{Y}_\infty\right]$ (Proposition~\ref{proposition:laplace-transforms}) and (ii) $\QQ_y\left[e^{-\beta (T_0-L)}\vert \overline{Y}_\infty,\jump\right]$ (Proposition~\ref{proposition:laplace-of-residual-time}): specifically, these Laplace transforms are given as algebraic expressions involving certain power series, whose coefficients are expressed directly in terms of the Laplace exponent of $X$. In addition, (iii)  the law of $\jump$ can be identified (Proposition~\ref{proposition:law-of-jump}). This then yields an explicit conditional, given $\overline{Y}_\infty$ and $\jump$, factorization of $T_0$ at the maximum, see Theorem~\ref{theorem}, which is our main result; and, because the law of $\overline{Y}_\infty$ is also known (see Subsection~\ref{subsection:first-passage} below), it characterizes the joint quadruple law of $(L,\overline{Y}_\infty,\jump,T_0-L)$. 

\begin{question}\label{section:concluding}
Can a suitable tractable conditional factorization/joint law be obtained if one adds into the mix also $Y_{T_0-}/(J\overline{Y}_\infty)$, the multiplicative jump at absorption relative to the position at the maximum? This is left open.
\end{question}

 Literature-wise we note that a different kind of (unconditional) Wiener-Hopf type factorization of the exponential functional of L\'evy processes is considered in \cite{pardo}. 



The organization of the remainder of this paper is as follows. Section~\ref{section:preliminaries} recalls the relevant fluctuation theory of snLp and introduces further necessary  notation in parallel. Then in Section~\ref{section:Laplace-transform} we develop the Laplace transform of the absorption time of $Y$ on the event that $Y$ does not go above a given level, which is a key result for the investigations of Section~\ref{section:main}, in which we finally present the laws of quantities at the maximum as announced in (i)-(ii)-(iii) above. Section~\ref{section:concluding} concludes with some remarks on (possible) applications. 

\section{Preliminaries and further notation concerning the snLp $X$}\label{section:preliminaries}

\subsection{Some basic fluctuation theory facts}\label{subsection:first-passage} 
The following is standard; we provide some specific references for the reader's benefit. 

(i). Let $\psi$ be the Laplace exponent of $X$, $\psi(\lambda):=\log \PP[e^{\lambda X_1}]$ for $\lambda\in [0,\infty)$. It has the representation 
\begin{equation}\label{eq:laplace-exp}
\psi(\lambda)=\frac{\diffusion}{2}\lambda^2+\mu\lambda+\int(e^{\lambda y}-\mathbbm{1}_{[-1,0)}(y)\lambda y-1)\nu(dy),\quad \lambda\in [0,\infty),
\end{equation}
for some (unique) $\mu\in \mathbb{R}$, $\diffusion\in [0,\infty)$, and measure $\nu$ on $\mathcal{B}_\mathbb{R}$, supported by $(-\infty,0)$, and satisfying $\int (1\land y^2 )\nu(dy)<\infty$. When $X$ has paths of finite variation, equivalently $\diffusion=0$ and $\int (1\land \vert y\vert)\nu(dy)<\infty$, we set $\drift:=d+\int_{[-1,0)}\vert y\vert\nu(dy)$; in this case we must have $\drift\in (0,\infty)$ and $\nu$ non-zero. Differentiating under the integral sign in \eqref{eq:laplace-exp}, $\psi$ is seen to be strictly convex, continuous,  with $\lim_\infty\psi=\infty$, and indeed with $\lim_\infty \psi'=\infty$ provided $X$ has paths of infinite variation.

We also let $\Phi$ be the right-continuous inverse of $\psi$, $\Phi(q):=\inf\{\lambda\in [0,\infty):\psi(\lambda)>q\}$ for $q\in [0,\infty)$, so that $\Phi(0)$ is the largest zero of $\psi$. Recall $X$ drifts to $\infty$, oscillates or drifts to $-\infty$, according as $\psi'(0+)>0$, $\psi'(0+)=0$ or $\psi'(0+)<0$ (the latter being equivalent to $\Phi(0)>0$). 

\begin{remark}
Assumption~\ref{assumption} means that $\Phi(p)>0$.
\end{remark}

(ii). For real $x\leq a$, $q\in [0,\infty)$, we have the classical identity \cite[Eq.~(3.15)]{kyprianou}
\begin{equation}\label{eq:classical}
\PP_x[e^{-q\tau_a^+};\tau_a^+<\infty]=e^{-\Phi(q)(a-x)},
\end{equation}
where $\tau_a^+:=\inf\{t\in (0,\infty):X_t>a\}$. \eqref{eq:classical} renders $\overline{X}_\expp-X_0$ to have the exponential distribution of rate $\Phi(p)$.

(iii). Associated to the solution \eqref{eq:classical} of the first passage upwards problem is the family, in $\lambda\in[0,\infty)$, of exponential $\FF$-martingales $\mathcal{E}^\lambda=(\mathcal{E}^\lambda_t)_{t\in [0,\infty)}$, 
\begin{equation} \label{eq:exp-mtgs}
\mathcal{E}^\lambda_t=e^{\lambda(X_t-X_0)-\psi(\lambda)t},\quad t\in[0,\infty). 
\end{equation}

(iv). Concerning the position of $X$ at first passage upwards, one has the resolvent identity \cite[Theorem~2.7(ii)]{kkr}: for real $x\leq a$, $q\in [0,\infty)$, and $f\in \mathcal{B}_{\mathbb{R}}/\mathcal{B}_{[0,\infty]}$,
\begin{equation}\label{eq:resolvent}
\int_0^\infty e^{-q t}\PP_x[f(X_t), t<\tau_a^+]dt=\int_{-\infty}^af(y)\left( e^{-\Phi(q)(a-x)}\Wm(a-y)-\Wm(x-y) \right)dy,
\end{equation}
where, for $q\in [0,\infty)$, $\Wm:\mathbb{R}\to [0,\infty)$ is the $q$-scale function of $X$, characterized by being continuous on $[0,\infty)$, vanishing on $(-\infty,0)$, and having Laplace transform 
\begin{equation}\label{eq:laplace}
\int_0^\infty e^{-\theta x}\Wm(x)dx=\frac{1}{\psi(\theta)-q},\quad \theta\in (\Phi(q),\infty).
\end{equation}
The reader is referred to \cite{kkr} for further background on scale functions of snLp. As usual we set $W^{(0)}=:W$ and recall that $W(0)>0$ or $W(0)=0$ according as $X$ has paths of finite or infinite variation \cite[Lemma~3.1]{kkr}. 

%

\subsection{Wiener-Hopf factorization}\label{wiener-hopf} 
The following falls under the umbrella of the Wiener-Hopf factorization.

(i). Because $X$ is regular upwards \cite[p.~232]{kyprianou} the process $(X_t;t\in [0,G))$ is independent of the process $(X_{G+t}-X_{G-};t\in [0,\expp-G))$  \cite[Lemma~VI.6(ii)]{bertoin}, and \cite[comment following Lemma~VI.6(ii)]{bertoin} $X_{G-}=X_G$ a.s. if and only if $X$ is regular downwards, i.e. $X$ has paths of infinite variation \cite[p.~232]{kyprianou}. In particular, the Wiener-Hopf factors $(G,\overline{X}_\expp)$ and $(\expp-G,\overline{X}_\expp-X_\expp)$ are independent. When $p>0$, then their Laplace transforms are given, for $\{\gamma,\delta\}\subset [0,\infty)$,  by \cite[Theorem~6.15(ii)]{kyprianou} $\PP[e^{-\gamma G-\delta \overline{X}_\expp}]=\frac{\kappa(p,0)}{\kappa(p+\gamma,\delta)}$ and $\PP[e^{-\gamma (\expp-G)+\delta (X_\expp-\overline{X}_\expp})]=\frac{\hat{\kappa}(p,0)}{\hat{\kappa}(p+\gamma,\delta)}$, where $\kappa$ and $\hat{\kappa}$ are the Laplace exponents of the increasing and decreasing ladder heights processes, respectively. The latter are themselves in turn expressed explicitly as  \cite[Subsection~6.5.2]{kyprianou}  $\kappa(\gamma,\delta)=\Phi(\gamma)+\delta$, $\hat{\kappa}(\gamma,\delta)=\frac{\gamma-\psi(\delta)}{\Phi(\gamma)-\delta}$, $\{\gamma,\delta\}\subset [0,\infty)$ (the expression for $\hat{\kappa}$ being understood in the limiting sense when $\Phi(\gamma)=\delta$). 

(ii). We have $\PP(\overline{X}_\expp=X_\expp)=\lim_{\beta\to\infty}\PP [e^{\beta(X_\expp-\overline{X}_\expp)}]=\lim_{\beta\to\infty}\frac{p(\Phi(p)-\beta)}{\Phi(p)(p-\psi(\beta))}$, so that 
\begin{equation}\label{eq:at-sup}
\PP(\overline{X}_\expp=X_\expp)\text{ is }=\frac{p}{\Phi(p)\drift}\text{ or is }=0,
\end{equation}
according as $X$ has paths of finite or infinite variation. For convenience we shall understand $\drift=\infty$ and (hence) $\frac{p}{\Phi(p)\drift}=0$ when $X$ has paths of infinite variation. Note also that $\frac{p}{\Phi(p)\drift}<1$.
%

\subsection{Excursions from the maximum} 
We gather here some facts concerning the It\^o point process  of excursions \cite{ito,blumenthal} from the maximum of $X$. For what follows, besides the general references given in the Introduction, the reader may also consult \cite{greenwood-pitman,rogers} \cite[passim]{fkp}. 

(i). Under $\PP$ the running supremum $\overline{X}$ serves as a continuous local time for $X$ at the maximum. Its right-continuous inverse is the process of first passage times $\tau^+=(\tau_a^+)_{a\in [0,\infty)}$. The time axis $[0,\infty)$ is partitioned $\PP$-a.s. into $\mm:=\overline{\{t\in [0,\infty):\overline{X}_t=X_t\}}=\{t\in [0,\infty):\overline{X}_t=X_t\text{ or }\overline{X}_{t}=X_{t-}\}$, the closure of the random set of times when $X$ is at its running supremum, and the open intervals $(\tau_{a-}^+,\tau_a^+)$, $a\in \dd:=\{b\in (0,\infty):\tau_{b-}^+<\tau_b^+\}$; the visiting set $\mm$ has $\PP$-a.s. no isolated points.

(ii). The process $\epsilon=(\epsilon_a)_{a\in (0,\infty)}$ defined for $a\in (0,\infty)$ by $$\epsilon_a(u):=X_{(\tau_{a-}^++u)\land \tau_a^+}-X_{\tau_{a-}^+-}=X_{(\tau_{a-}^++u)\land \tau_a^+}-a,\quad u\in [0,\infty),$$ if $a\in \dd$, $\epsilon_a:=\Delta$ otherwise, where $\Delta\notin \DD$ is a coffin state, is, under $\PP$, a Poisson point process (Ppp) with values in $(\DD,\DDD)$, in the filtration $\FF_{\tau^+}=(\FF_{\tau^+_a})_{a\in [0,\infty)}$, absorbed on first entry into a path $\omega\in \DD$ for which $\zeta(\omega)=\infty$, where $\zeta(\omega):=\inf\{t\in (0, \infty):\omega(t)\geq 0\}$, and whose characteristic measure we will denote by $\nn$ (so the intensity measure of $\epsilon$ is $\mathsf{l}\times \nn$, where $\mathsf{l}$ is Lebesgue measure on $\mathcal{B}_{(0,\infty)}$). Note that $\nn$ is carried by the set $\{\zeta>0\}$ and also by $\{\xi_0<0\}$ (resp. $\{\xi_0=0\}$) when $X$ has paths of finite (resp. infinite) variation. 
Besides, $\PP$-a.s., for all $a\in \dd$, $\zeta(\epsilon_a)=\tau_{a}^+-\tau_{a-}^+$. 

(iii). The compensation formula for Ppp tells us that $\PP \left[\sum_{a\in \dd}Z_{a}(\epsilon_a)\right]=\PP\left[\int_0^{\overline{X}_\infty}\int Z_a(\omega)\nn(d\omega)da\right]$ whenever $Z\in \PPP_{\FF_{\tau^+}}\otimes \DDD/\mathcal{B}_{[0,\infty]}$, where $ \PPP_{\FF_{\tau^+}}$ is the $\FF_{\tau^+}$-predictable $\sigma$-field. In particular note that if $R\in \PPP_\FF\otimes \DDD/\mathcal{B}_{[0,\infty]}$, where $\PPP_\FF$ is now the $\FF$-predictable $\sigma$-field, then $(R_{\tau_{a-}^+})_{a\in (0,\infty)}\in \PPP_{\FF_{\tau^+}}\otimes \DDD/\mathcal{B}_{[0,\infty]}$. 

(iv). The measure $\nn$ has the following Markov property: for all $t\in [0,\infty)$, $$\nn[G\cdot  H\circ \theta_t;t<\zeta]=\nn[G\PP_{\xi_t}[H];t<\zeta],\quad G\in \DDD_t/\mathcal{B}_{[0,\infty]},H\in \DDD/\mathcal{B}_{[0,\infty]}.$$

(v). In case $X$ has paths of infinite variation (equivalently, $X$ is regular downwards) the result of \cite[Corollary~1]{chaumont-doney} applied to $-X$ (in conjunction with  \cite[Eqs.~(2.5) and~(2.8)]{chaumont-doney} therein, \eqref{eq:classical}, and the fact that in this case $\nn(1-e^{-p\zeta}\mathbbm{1}_{\{\zeta<\infty\}})=\Phi(p)$, see Remark~\ref{remark:law-of-zeta})  implies that for all $t\in (0,\infty)$ and then all $F\in \DDD_t/\mathcal{B}_{\mathbb{R}}$ bounded and continuous in the Skorokhod topology on $\DD$,
\begin{equation}\label{eq:conditioning}
\nn\left[F;t<\zeta\right]=\kk\lim_{x\uparrow 0}\frac{\PP_x\left[F;t<\tau_0^+\right]}{\gg(x)},
\end{equation}  where $$\gg(x):=\lim_{q\downarrow 0}\frac{1-e^{\Phi(q)x}}{\Phi(q)}=
\begin{cases}
\frac{1-e^{\Phi(0)x}}{\Phi(0)}, & \Phi(0)>0\\
-x, & \Phi(0)=0
\end{cases},\quad x\in (-\infty,0),
$$
and with $\kk\in (0,\infty)$ depending on the characteristics of $X$ only. 

(vi). One has the following representation of the scale function $W$ \cite[Eq.~(31)]{kkr}:
\begin{equation}\label{eq:scale-excursion}
\frac{W(x)}{W(a)}=\exp\left\{-\int_x^a\nn(-\underline{\xi}_{\zeta}>y)dy\right\},\quad 0\leq x< a\text{ real},
\end{equation}
where $\underline{\xi}=(\underline{\xi}_t)_{t\in [0,\infty]}$ is the running infimum process of $\xi$. 

\subsection{Point process of jumps} Under $\PP$ the process $\Delta X=(\Delta X_t)_{t\in (0,\infty)}$ of the jumps of $X$ is a Ppp in the filtration $\FF$ with values in $(\mathbb{R}\backslash \{0\},\mathcal{B}_{\mathbb{R}\backslash \{0\}})$ and $0$ as a coffin state, whose characteristic measure is the restriction to $\mathcal{B}_{\mathbb{R}\backslash \{0\}}$ of $\nu$. In this case the compensation formula for Ppp states that $\PP\left[\sum_{t\in\jj}Z_{t}(\Delta X_t)\right]=\PP\left[\int_0^\infty \int Z_t(x)\nn(dx)dt\right]$ for $Z\in \PPP_\FF\otimes \mathcal{B}_{\mathbb{R}\backslash \{0\}}/\mathcal{B}_{[0,\infty]}$, where $\jj:=\{t\in (0,\infty):\Delta X_t\ne 0\}$ is the set of jump times of $X$. See for instance \cite[Theorem~I.1]{bertoin}.

\subsection{Patie's scale functions} We set, assuming $\Phi(p)\notin \alpha\mathbb{N}$,
$$\JJ(y):=\sum_{k=0}^{\infty}\aa_ky^k,\quad y\in [0,\infty),\text{ where }\aa_k:=\left(\prod_{l=1}^k(\psi(l\alpha)-p)\right)^{-1},\quad k\in \mathbb{N}_0,$$
(the condition on $\Phi(p)$ ensures all the $\aa_k$ are well-defined) and, whether or not $\Phi(p)\notin \alpha\mathbb{N}$, 
$$\II(y):=\sum_{k=0}^{\infty}\bb_k y^k,\quad y\in [0,\infty),\text{ where }\bb_k:=\left(\prod_{l=1}^k(\psi(\Phi(p)+l\alpha)-p)\right)^{-1},\quad k\in \mathbb{N}_0,$$
with (as usual) the empty product being interpreted as $=1$. These  power series converge absolutely, indeed $\lim_{k\to\infty}\frac{\aa_{k+1}}{\aa_k}=\lim_{k\to\infty}\frac{\bb_{k+1}}{\bb_k}=0$, and the coefficients $\aa_k$ are ultimately of the same sign (even all strictly positive when $\Phi(p)<\alpha$). Finally,
\begin{equation}\label{exit-1}
\PP_x\left[e^{-\gamma I_{\tau_a^+}};\tau_a^+<\expp\right]=\frac{e^{\Phi(p)x}\II(\gamma e^{\alpha x})}{e^{\Phi(p)a}\II(\gamma e^{\alpha a})},\quad x\leq a\text{ real},\gamma\in[0,\infty);
\end{equation}
in other words
\begin{equation}\label{exit}
\QQ_y\left[e^{-\gamma T_d^+};T_d^+<\infty\right]=\left(\frac{y}{d}\right)^{\Phi(p)}\frac{\II(\gamma y^\alpha)}{\II(\gamma d^\alpha)},\quad 0<y\leq d\text{ real},\gamma\in[0,\infty),
\end{equation}
where $T_d^+:=\inf\{s\in(0,\infty):Y_s>d\}$ for $d\in (0,\infty)$. See  \cite[Theorem~2.1]{pierre} (or  \cite[Section~13.7]{kyprianou}). 
\begin{remark}
$\psi-p$ is the Laplace exponent of the snLp $X$ killed (and sent to $-\infty$)  at time $\expp$. $\psi(\Phi(p)+\cdot)-p$ is the Laplace exponent of $X$ under the Esscher transform corresponding to the exponential martingale $\mathcal{E}^{\Phi(p)}$. In this sense $\JJ$ and $\II$ may both be viewed as being just two special instances of the same underlying power series that is in general associated to the Laplace exponent of a (possibly killed) snLp and an index $\alpha$. 
\end{remark}

\section{Laplace transform of the absorption time of $Y$ (on the event that $Y$ does not go above a given level)}\label{section:Laplace-transform}
\begin{definition}\label{definition:MM}
We introduce the function
\begin{equation}\label{eq:mm}
\MM(y,d):=\JJ(\beta y^\alpha)-\left(\frac{y}{d}\right)^{\Phi(p)}\frac{\II(\beta y^\alpha)}{\II(\beta d^\alpha)}\JJ(\beta d^\alpha),\quad \{y,d\}\subset (0,\infty),y\leq d, \beta\in [0,\infty),
\end{equation}
where the expression must be understood in the limiting sense (as $\alpha\to \Phi(p)/m$), when $\Phi(p)=\alpha m$ for some $m\in \mathbb{N}$: it will be seen from Corollary~\ref{corollary} that this limit exist a priori; and it is identified analytically in Remark~\ref{remark:limits:ii} to follow.
\end{definition}
\begin{remark}\label{remark:limits:ii}  
Suppose $\Phi(p)=m\alpha$ for an $m\in \mathbb{N}$. Then we may write, for $\alpha'\in (\frac{\Phi(p)}{m+1},\frac{\Phi(p)}{m-1})\backslash \{\frac{\Phi(p)}{m}\}$, setting  provisionally $\widetilde{\bbb_k}:=(\prod_{l=1}^k(\psi(m\alpha'+l\alpha')-p))^{-1}$ for $k\in \mathbb{N}_0$,

$$\MMM(y,d)=\sum_{k=0}^{m-1} \aaa_k(\beta y^{\alpha'})^k-\left(\frac{y}{d}\right)^{\Phi(p)}\frac{\III(\beta y^{\alpha'})}{\III(\beta d^{\alpha'})}\sum_{k=0}^{m-1} \aaa_k(\beta d^{\alpha'})^k$$
$$+\aaa_{m-1}\beta^m\frac{y^{m\alpha'}\sum_{k=0}^\infty \widetilde{\bbb_k}(\beta y^{\alpha'})^k-\left(\frac{y}{d}\right)^{\Phi(p)}\frac{\III(\beta y^{\alpha'})}{\III(\beta d^{\alpha'})}d^{m\alpha'}\sum_{k=0}^\infty \widetilde{\bbb_k}(\beta d^{\alpha'})^k}{\psi(m\alpha')-p}$$

$$\xrightarrow{\alpha'\to \frac{\Phi(p)}{m}}\sum_{k=0}^{m-1} \aa_k(\beta y^{\alpha})^k-\left(\frac{y}{d}\right)^{\Phi(p)}\frac{\II(\beta y^{\alpha})}{\II(\beta d^{\alpha})}\sum_{k=0}^{m-1} \aa_k(\beta d^{\alpha})^k$$
$$+\frac{\aa_{m-1}(\beta y^\alpha)^m}{\psi'(\Phi(p))}\Big[\ln\left(\frac{y}{d}\right)\II(\beta y^{\alpha})-\KK(\beta y^\alpha)+\frac{\II(\beta y^\alpha)}{\II(\beta d^\alpha)}\KK(\beta d^{\alpha})\Big]=\MM(y,d),$$ 
where
$$\KK(z):=\sum_{k=1}^\infty \cc_kz^k,\quad z\in (0,\infty),$$
with $$\cc_k:=\bb_k\sum_{l=1}^k\frac{\psi'(\Phi(p)+l\alpha)}{\psi(\Phi(p)+l\alpha)},\quad k\in \mathbb{N}$$
(it is easy to check that $\psi'/\psi$ is bounded on $[c,\infty)$ for any $c\in (\Phi(0),\infty)$).
\end{remark}
\begin{remark}
It is clear from the definition of $\JJ$ and $\II$ (and from Remark~\ref{remark:limits:ii}) in the case when $\Phi(p)\notin \alpha\mathbb{N}$ (when $\Phi(p)\in \alpha\mathbb{N}$) that $\mathcal{M}^{p,\alpha}_\cdot(\cdot,\cdot)$ is jointly continuous. This also follows in any case from \eqref{eq:exit-conditioned} below:  by quasi left-continuity and regularity of $0$ for $(0,\infty)$ of the
 process $X$ and from the distribution of $\overline{X}_\expp$ not having any finite atoms, one concludes that for each $c\in \mathbb{R}$, a.s. $\cap_{d\in (-\infty,c)}\{\tau_d^+< \expp\}=\{\tau_c^+< \expp\}=\cup_{d\in (c,\infty)}\{\tau_d^+< \expp\}$, so that bounded convergence applies in \eqref{eq:exit-conditioned} (once one has passed from $\PP_x$ to $\PP$ via spatial homogeneity of $X$). Furthermore, the relation $0\leq \MM(y,d)\leq 1-(\frac{y}{d})^{\Phi(p)}$ will also follow directly from \eqref{eq:exit-conditioned} (via \eqref{eq:classical}).
\end{remark}

\begin{proposition}\label{proposition:mtgs-exit}
Assume $\Phi(p)\notin \alpha\mathbb{N}$. Let $\beta\in [0,\infty)$ and set
$$M_t:=e^{-\beta I_{t\land \expp}}\JJ(\beta e^{\alpha X_{t}}\mathbbm{1}_{\{t<\expp\}}),\quad t\in [0,\infty),$$
and 
$$N_s:=e^{-\beta (s\land T_0)}\JJ(\beta Y_s^\alpha),\quad s\in [0,\infty).$$
 Then:
\begin{enumerate}[(i)]
\item\label{prelim:i} For each $y\in (0,\infty)$, under $\QQ_y$, $N$ is a martingale in the filtration $\mathcal{G}$.
\item\label{prelim:ii} For each $c\in \mathbb{R}$ and  $x\in \mathbb{R}$, under $\PP_x$,
$M^{\tau_c^+}$ is a martingale in the filtration $\FF$.\footnote{$M^{\tau_c^+}$ is the process $M$ stopped at $\tau_c^+$.}
\end{enumerate}
\end{proposition}
The proof of Proposition~\ref{proposition:mtgs-exit}, and also of Corollary~\ref{corollary} below, follows on p.~\pageref{proof}. 
\begin{remark}\label{remark:other-mtgs}
One has the following parallel statements. For $\gamma\in [0,\infty)$:
\begin{enumerate}[(I)]
\item\label{I} for each $d\in (0,\infty)$ and $y\in (0,d]$, under $\QQ_y$, the process $(e^{-\gamma s}Y_s^{\Phi(p)}\II(\gamma Y_s^\alpha))_{s\in [0,\infty)}$ stopped at $T_d^+$ is a martingale in $\GG$ with terminal value $d^{\Phi(p)}\II(\gamma d^\alpha)e^{-\gamma T_d^+}\mathbbm{1}_{\{T_d^+<T_0\}}$;
\item\label{II} for each $a\in \mathbb{R}$ and  $x\in \mathbb{R}$, under $\PP_x$, the process $(e^{-\gamma I_t}e^{\Phi(p)X_t}\II(\gamma e^{\alpha X_t})\mathbbm{1}_{\{t<\expp\}})_{t\in [0,\infty)}$ stopped at $\tau_a^+$ is a martingale in $\FF$ with terminal value $e^{\Phi(p)a}\II(\gamma e^{\alpha a})e^{-\gamma I_{\tau_a^+}}\mathbbm{1}_{\{\tau_a^+<\expp\}}$.
\end{enumerate}
Indeed \ref{I} is a direct consequence of \eqref{exit} and the Markov property of $Y$ in $\GG$, whereas it is perhaps easiest to get \ref{II} by optional sampling on the martingale from \ref{I} via the time change $I=(I_t)_{t\in [0,\infty)}$ (in an analogous manner as we will see in the first paragraph of the proof of Proposition~\ref{proposition:mtgs-exit}).
\end{remark}
 \begin{remark}
In the case $p=0=\Phi(0)$ ($T_0=\infty$ a.s.), the martingale claim on $N$ can be found in  \cite[Theorem~13.9]{kyprianou}. The case $\Phi(0)>0=p$ is inspired by \cite[Eq. (2.4)]{pierre}. Indeed the latter result implies that, if further Rivero's condition \cite[Theorem~2]{rivero} $\Phi(0)\in (0,\alpha)$ (which guarantees $Y$ admits a self-similar recurrent extension that leaves $0$ continuously) is met, then for some $b\in (0,\infty)$, the process $M-bO$, where $O:=(e^{\Phi(0) X_t}\mathcal{I}^{0,\alpha}(\beta e^{\alpha X_t}))_{t\in [0,\infty)}$, is a (even bounded nonnegative) martingale in $\FF$. Combined with \ref{II} of Remark~\ref{remark:other-mtgs}, \ref{prelim:ii}  follows (apply optional stopping and the fact that martingales form a linear space). The general case can be handled by the same techniques as are those that are used in the proof of  \cite[Theorem~13.9]{kyprianou},  but the details are quite delicate, so we provide an explicit proof below.  More generally, Proposition~\ref{proposition:mtgs-exit} is related to the computation of the positive entire moments of (spectrally negative) pssMp, for which see \cite{bertoin-yor} and more recently \cite{doring-barczy}. We mention also the paper \cite{patie} that provides some analytical representations of the density of the absorption time $T_0$. 
\end{remark}


\begin{corollary}\label{corollary}
For real $x\leq c$ and $\beta\in [0,\infty)$,
\begin{equation}\label{eq:exit-conditioned}
\PP_x\left[e^{-\beta I_\expp}; \tau_c^+\geq \expp\right]=\mathcal{M}^{p,\alpha}_\beta(e^x,e^c).
\end{equation}
\end{corollary}

\begin{remark}
Let, provisionally, $f(\beta):=\PP[e^{-\beta I_\expp}]$ for $\beta\in [0,\infty)$. Assume $\Phi(p)\notin \alpha\mathbb{N}$. Then, by \eqref{eq:exit-conditioned} and \eqref{exit-1}, for all real $x\leq c$, $f(\beta e^{\alpha x})=\PP_x[e^{-\beta I_{\tau_c^+}};\tau_c^+<\expp]\PP_c[e^{-\beta I_\expp}]+\PP_x[e^{-\beta I_\expp};\tau_c^+\geq \expp]=\frac{e^{\Phi(p)x}\II(\beta e^{\alpha x})}{e^{\Phi(p)c}\II(\beta e^{\alpha c})}f(\beta e^{\alpha c})+\JJ(\beta e^{\alpha x})-\frac{e^{\Phi(p)x}\II(\beta e^{\alpha x})}{e^{\Phi(p)c}\II(\beta e^{\alpha c})}\JJ(\beta e^{\alpha c})$, i.e. $\frac{f(\beta e^{\alpha x})-\JJ(\beta e^{\alpha x})}{(\beta e^{\alpha x})^{\frac{\Phi(p)}{\alpha}}\II(\beta e^{\alpha x})}=\frac{f(\beta e^{\alpha c})-\JJ(\beta e^{\alpha c})}{(\beta e^{\alpha c})^{\frac{\Phi(p)}{\alpha}}\II(\beta e^{\alpha c})}$. It follows that for some $b_{p,\alpha}\in \mathbb{R}$ (the $b_{p,\alpha}$ of course also depends on the characteristic of $X$; we make explicit only the dependence on $\alpha$ and $p$), and then all $x\in \mathbb{R}$, $\beta \in [0,\infty)$, one has
\begin{equation}\label{eq:extension}
\PP_x[e^{-\beta I_\expp}]=f(\beta e^{\alpha x})=\JJ(\beta e^{\alpha x})-b_{p,\alpha} (\beta e^{\alpha x})^{\frac{\Phi(p)}{\alpha}}\II(\beta e^{\alpha x}).
\end{equation} Since by bounded convergence, for $\beta>0$, $\lim_{x\to\infty}\PP_x[e^{-\beta I_\expp}]=0$, one can identify $C_{p,\alpha}$ as the unique real number, necessarily not zero\footnote{Indeed strictly positive or strictly negative according as the coefficients $\aa_n$ are ultimately all strictly positive or strictly negative.}, for which $\lim_{y\to\infty}(\JJ(y)- C_{p,\alpha} y^{\frac{\Phi(p)}{\alpha}}\II(y))=0$. Again when $\Phi(p)=\alpha m$ for some $m\in \mathbb{N}$, then \eqref{eq:extension} still holds, provided the right-hand side is understood in the limiting sense as $\alpha\to \Phi(p)/m$. This generalizes the result of  \cite[Eq. (2.4)]{pierre} to the case when $p>0$ or else $\Phi(0)\in [\alpha,\infty)$. 
\end{remark}
\begin{proof}[Proofs of Proposition~\ref{proposition:mtgs-exit} and Corollary~\ref{corollary}]\label{proof}
Suppose \ref{prelim:i} has been established. For each $c\in \mathbb{R}$, $M^{\tau_c^+}$ is $N^{T_{e^c}^+}$, time changed by $I=(I_t)_{t\in [0,\infty)}$. By optional stopping $N^{T_{e^c}^+}$ is an a.s. bounded martingale in $\GG$. $I$ is a family of finite $\GG$-stopping times. Thus, by optional sampling on $N^{T_{e^c}^+}$, the martingale property of $M^{\tau_c^+}$ follows (note that $\FF_t\subset \GG_{I_t}$ for all $t\in [0,\infty)$). Assuming $\Phi(p)\notin \alpha\mathbb{N}$, Corollary~\ref{corollary} then obtains yet again by optional sampling, this time on the martingale $M^{\tau_c^+}$: for real $x\leq c$, $\beta\in[0,\infty)$, one has $\PP_x\left[e^{-\beta I_{\tau_c^+}}\JJ(\beta e^{\alpha c});\tau_c^+<\expp\right]+\PP_x\left[e^{-\beta I_\expp}; \tau_c^+\geq \expp\right]=\JJ(\beta e^{\alpha x})$, followed by an application of \eqref{exit-1}. The case $\Phi(p)\in \alpha\mathbb{N}$ is got by taking limits.

So it remains to argue \ref{prelim:i}. Let $s\in[0,\infty)$, $y\in (0,\infty)$ and $n\in \mathbb{N}$.  

Recall first from \eqref{exit} that for $\gamma\in (0,\infty)$ and $d\in [y,\infty)$, $\QQ_y[e^{-\gamma T_d^+};T_d^+<\infty]=\left(\frac{y}{d}\right)^{\Phi(p)}\frac{\II(\gamma y^\alpha)}{\II(\gamma d^\alpha)}$. But $\QQ_y[e^{-\gamma T_d^+};T_d^+<\infty]=\int_0^\infty \gamma e^{-\gamma s}\QQ_y(\overline{Y}_s>d)ds$. 
Hence, for $k\in (0,\infty)$, using $\QQ_y[\overline{Y}_s^k]=k\int_0^\infty m^{k-1}\QQ_y(\overline{Y}_s>m)dm$ and Tonelli's theorem,
\begin{equation}\label{eq:app:moments}
\int_0^\infty \gamma e^{-\gamma s}\QQ_y[\overline{Y}_s^k]ds=y^k+k\int_y^\infty m^{k-1}\left(\frac{y}{m}\right)^{\Phi(p)}\frac{\II(\gamma y^\alpha)}{\II(\gamma m^\alpha)}dm<\infty
\end{equation} (the finiteness is clear from the definition of $\II$). We conclude that $\QQ_y[\overline{Y}_s^n]<\infty$.\footnote{Incidentally, formula \eqref{eq:app:moments} gives the positive  moments of $\overline{Y}$ sampled at an independent exponential random time of rate $\gamma$ (it is even trivially valid for $k=0$).}

We have next, for $t\in [0,\infty)$, that $\QQ_y[Y_{s\land I_t}^{\alpha n}]=\PP_{\log y}[e^{\alpha nX_{\phi_s\land t}};\phi_s\land t<\expp]=\QQ_y[e^{\alpha nX_{\phi_s\land t}-p(\phi_s\land t)}]$, because $\expp$ is independent of $X$; note that $\phi_{s\land   I_t}=\phi_s\land  t$. 
Furthermore, because $\phi$ is the right inverse of $I$, which is, by the fundamental theorem of calculus, differentiable from the right and differentiable at every continuity point of $X$, 
$$\frac{d^+\phi_v}{dv}=e^{-\alpha X_{\phi_v}},\quad v\in [0,I_\infty),$$ where $d^+$ signifies that the right-derivative is meant, with the understanding that it can be replaced by the ordinary derivative at every $v\in [0,I_\infty)$ for which $\phi_v$ is a continuity point of $X$ (and hence for Lebesgue-almost every (indeed all except countably many) $v\in [0,I_\infty)$). Then, with analogous provisos, $$\frac{d^+}{dv}e^{(\psi(\alpha n)-p)\phi_v}=e^{(\psi(\alpha n)-p)\phi_v}(\psi(\alpha n)-p)e^{-\alpha X_{\phi_v}},\quad v\in [0,I_\infty),$$ so that 
$$e^{(\psi(\alpha n)-p)\phi_v}=1+(\psi(\alpha n)-p)\int_0^ve^{(\psi(\alpha n)-p)\phi_w-\alpha X_{\phi_w}}dw,\quad v\in [0,I_\infty)$$
(the function $[0,\infty)\ni u\mapsto X_u$ is locally bounded away from $-\infty$, hence $[0,I_\infty)\ni v\mapsto \phi_v$ is locally Lipschitz and thus absolutely continuous; accordingly the fundamental theorem of calculus applies). Then $$e^{(\psi(\alpha n)-p)\phi_{s\land I_t}}=1+(\psi(\alpha n)-p)\int_0^se^{(\psi(\alpha n)-p)\phi_v-\alpha X_{\phi_v}}\mathbbm{1}_{\{\phi_v<t\}}dv.$$
Now multiply both sides by $e^{\alpha n (X_{\phi_{s\land I_t}}-\log y)-\psi(\alpha n)\phi_{s\land I_t}}$ and take the $\QQ_y$-expectation. By Tonelli's theorem, by optional sampling on the exponential martingale $\mathcal{E}^{\alpha n}$ \eqref{eq:exp-mtgs} at the bounded $\FF$-stopping times $\phi_s\land t$ and $\phi_v\land t$, and by the independence of $X$ from $\expp$, it follows that 
 $$\QQ_y[Y_{s\land I_t}^{\alpha n}]=y^{\alpha n}+(\psi(\alpha n)-p)\int_0^s\QQ_y[e^{\alpha(n-1)X_{\phi_{v}}-p\phi_v};v<I_t]dv$$
$$=y^{\alpha n}+(\psi(\alpha n)-p)\int_0^s\QQ_y[Y_{v}^{\alpha (n-1)};v<I_t,\phi_v<\expp]dv.$$
Now let $t\to \infty$. By dominated convergence on the left-hand side and monotone convergence on the right-hand side, and because $Y$ is constant on $[I_\expp,\infty)\supset [I_\infty,\infty)$ and a.s continuous at $I_\infty$ when $p=0$, we obtain
 
\begin{equation}\label{preliminaries:1}
\QQ_y[Y_{s}^{\alpha n}]=y^{\alpha n}+(\psi(\alpha n)-p)\int_0^s\QQ_y[Y_{v}^{\alpha (n-1)};v<T_0]dv.
\end{equation}

It is now proved by induction that $\QQ_y[Y_s^{\alpha n}]\leq \sum_{k=0}^n\frac{\vert \aa_{n-k}\vert }{\vert \aa_n\vert }y^{\alpha(n-k)}\frac{s^k}{k!}$ (with equality when $T_0=\infty$ a.s., i.e. $\Phi(0)= 0=p$, in which case all the $\aa_l$ are positive). Then
$$\QQ_y[e^{-\beta s}\vert \JJ(\beta Y_s^\alpha)\vert ]\leq e^{-\beta s}\sum_{n=0}^\infty \vert \aa_n\vert \beta^n\QQ_y[Y_s^{\alpha n}]$$
(with equality when $\Phi(0)=0=p$, by Tonelli's theorem)
$$\leq e^{-\beta s}\sum_{n=0}^\infty \vert \aa_n\vert \beta^n\sum_{k=0}^n\frac{\vert \aa_{n-k}\vert }{\vert \aa_n\vert}y^{\alpha(n-k)}\frac{s^k}{k!}=e^{-\beta s}\sum_{k=0}^\infty\frac{(\beta s)^k}{k!}\sum_{n=k}^\infty\vert \aa_{n-k}\vert  (\beta y^\alpha)^{n-k}= \sum_{n=0}^\infty\vert \aa_{n}\vert  (\beta y^\alpha)^{n}<\infty $$
(with equality, and $=\JJ(\beta y^\alpha)$ when $\Phi(0)=0=p$).
We conclude that 
\begin{equation}\label{eq:loc-bdd}
\sup_{v\in [0,\infty)}\QQ_y[e^{-\beta v}\vert \JJ(\beta Y_v^\alpha)\vert ]<\infty
\end{equation}
 (and $\QQ_y[e^{-\beta s} \JJ(\beta Y_s^\alpha)]=\JJ(\beta y^\alpha)$ when $\Phi(0)=0=p$). In the case when $\Phi(0)=0=p$ ($T_0=\infty$ a.s.) it is now already  standard to argue that $N$ is a martingale in $\GG$, but to handle the general scenario we have to do a little more work. 

Specifically, we show that 
\begin{equation}\label{preliminaries:2}
\QQ_y[e^{-\beta(s\land T_0)}\JJ(\beta Y_{s}^\alpha)]=\JJ(\beta y^\alpha).
\end{equation}
By self-similarity we may assume $y=1$. We then have that
$$f(s):=\QQ[e^{-\beta(s\land T_0)}\JJ(\beta Y_{s}^\alpha)]=\QQ[e^{-\beta T_0};T_0\leq s]+ e^{-\beta s}\sum_{n=0}^\infty \aa_n\beta^n\QQ[Y_s^{\alpha n};s<T_0]$$
(by linearity and Tonelli's theorem, recalling all the $\aa_n$ are ultimately of the same sign)
 $$=\QQ[e^{-\beta (s\land T_0)}]+e^{-\beta s}\sum_{n=1}^\infty \aa_n\beta^n\QQ[Y_s^{\alpha n}],$$
which by \eqref{preliminaries:1} 
$$= \QQ[e^{-\beta (s\land T_0)}]+e^{-\beta s}\sum_{n=1}^\infty \aa_n\beta^n\left(1+(\psi(\alpha n)-p)\int_0^s\QQ_y[Y_{v}^{\alpha (n-1)};v<T_0]dv\right)$$
$$=\QQ[e^{-\beta (s\land T_0)}]+e^{-\beta s}(\JJ(\beta )-1)+e^{-\beta s}\beta \int_0^s\sum_{n=0}^\infty \aa_n\beta^n\QQ[Y_v^{\alpha n};v<T_0]dv$$
(again the interchange of the integral and summation is justified by the fact that all the $\aa_l$ are ultimately of the same sign)
$$=\QQ[e^{-\beta (s\land T_0)}]+e^{-\beta s}(\JJ(\beta )-1)+e^{-\beta s}\beta\int_0^se^{\beta v}(f(v)-\QQ[e^{-\beta T_0};T_0\leq v])dv$$ 
$$=\QQ[e^{-\beta (s\land T_0)}]+e^{-\beta s}(\JJ(\beta )-1)+e^{-\beta s}\beta\int_0^se^{\beta v}f(v)dv-\QQ[e^{-\beta T_0}-e^{-\beta s};s\geq T_0]$$ 
$$=\JJ(\beta)e^{-\beta s}+\beta e^{-\beta s}\int_0^se^{\beta v}f(v)dv.$$
We now have the integral equation $f(s)=\JJ(\beta)e^{-\beta s}+\beta e^{-\beta s}\int_0^se^{\beta v}f(v)dv$ for $f:[0,\infty)\to[0,\infty)$. But $f$ is locally bounded because of \eqref{eq:loc-bdd}, hence continuous by bounded convergence, hence continuously differentiable by the fundamental theorem of calculus. Differentiating we obtain $f'=0$ and thus $f(s)=f(0)=\JJ(\beta)$, as was to be shown.

With \eqref{preliminaries:2} having been established, showing that $N$ is a martingale in $\GG$, is an exercise in applying the Markov property of $Y$ on $[0,T_0)$: for $\{s_1,s_2\}\subset [0,\infty)$, 
$$\QQ_y[e^{-\beta((s_1+s_2)\land T_0)}\JJ(\beta Y_{s_1+s_2}^\alpha)\vert \GG_{s_1}]$$ 
$$=\mathbbm{1}_{\{s_1<T_0\}}e^{-\beta s_1}\QQ_{Y_{s_1}}[e^{-\beta (s_2\land T_0)}\JJ(\beta Y_{s_2}^\alpha)]+\mathbbm{1}_{\{s_1\geq T_0\}}e^{-\beta T_0}=e^{-\beta(s_1\land T_0)}\JJ(\beta Y_{s_1}^\alpha),$$
a.s.-$\QQ_y$, where we used \eqref{preliminaries:2} in the last equality.
\end{proof}

\section{Laws of quantities at the maximum}\label{section:main}

\subsection{Laplace transform of $L$ given $\overline{Y}_\infty$}

\begin{proposition}\label{proposition:laplace-transforms}
Let $\gamma\in [0,\infty)$. For $x\in \mathbb{R}$, $$\PP_x\left[\exp\left\{-\gamma\int_0^Ge^{\alpha X_u}du\right\}\Big\vert \overline{X}_\expp\right]=\frac{\II(\gamma e^{\alpha x})}{\II(\gamma e^{\alpha \overline{X}_\expp})}$$ a.s.-$\PP_x$.; in other words, for $y\in (0,\infty)$, a.s.-$\QQ_y$,
$$\QQ_y\left[e^{-\gamma L}\vert \overline{Y}_\infty\right]=\frac{\II(\gamma y^\alpha)}{\II(\gamma {\overline{Y}_\infty}^\alpha)}.$$ 
\end{proposition}
\begin{remark}\label{remark:law-of-zeta}
In the proof we will, en passant,  establish the identity $\nn(1-e^{-p\zeta}\mathbbm{1}_{\{\zeta<\infty\}})=\Phi(p)-\frac{p}{\drift}$ (recall we interpret $\drift=\infty$, hence $\frac{p}{\drift}=0$, when $X$ has paths of infinite variation).
\end{remark}
\begin{proof}
Without loss of generality we work under $\PP$: for $x\in \mathbb{R}$, the law of $(L,\overline{X}_\expp)$ under $\PP_x$ is that of $(e^{\alpha x}L,x+\overline{X}_\expp)$ under $\PP$.
Then we are to determine, for $f\in \mathcal{B}_\mathbb{R}/\mathcal{B}_{[0,\infty]}$,
$$\AA:=\PP\left[\exp\left\{-\gamma \int_0^G e^{\alpha X_u}du\right\}f(\overline{X}_\expp)\right]=$$\footnotesize
$$\PP \left[\sum_{a\in \dd}\exp\left\{-\gamma \int_0^{\tau_{a-}^+} e^{\alpha X_u}du\right\}f(\overline{X}_{\tau_{a-}^+})\mathbbm{1}_{\{\tau_{a-}^+<\expp\leq \tau_a^+\}}\right]+\PP\left[\exp\left\{-\gamma \int_0^G e^{\alpha X_u}du\right\}f(\overline{X}_G);\overline{X}_\expp=X_\expp\right]. $$\normalsize Here the second term only appears when $p>0$. From the Wiener-Hopf factorization the event $\{\overline{X}_\expp=X_\expp\}$ is independent of the process $X$ on the time interval $[0,G)$, and \eqref{eq:at-sup} $\PP(\overline{X}_\expp=X_\expp)=\frac{p}{\Phi(p) \drift}$. In consequence \footnotesize $$\AA\left(1-\frac{p}{\Phi(p)\drift}\right)=
\PP \left[\sum_{a\in \dd}\exp\left\{-p \tau_{a-}^+-\gamma \int_0^{\tau_{a-}^+} e^{\alpha X_u}du\right\}f(\overline{X}_{\tau_{a-}^+})\left(1-e^{-p(\tau_a^+-\tau_{a-}^+)}\mathbbm{1}_{\{\tau_a^+<\infty\}}\right)\right]$$\normalsize
$$=\PP \left[\sum_{a\in \dd}\exp\left\{-p \tau_{a-}^+-\gamma \int_0^{\tau_{a-}^+} e^{\alpha X_u}du\right\}f(\overline{X}_{\tau_{a-}^+})\left(1-e^{-p\zeta(\epsilon_a)}\mathbbm{1}_{\{\zeta(\epsilon_a)<\infty\}}\right)\right].$$
Now by the absence of positive jumps $\overline{X}$ is continuous and so predictable in $\FF$.  The compensation formula for $\epsilon$ thus entails
$$\AA\left(1-\frac{p}{\Phi(p)\drift}\right)=\PP\left[\int_0^{\overline{X}_\infty} e^{-p\tau_{a-}^+-\gamma\int_0^{\tau_{a-}^+}e^{\alpha X_u}du}f(\overline{X}_{\tau_{a-}^+})da\right]\nn(1-e^{-p\zeta}\mathbbm{1}_{\{\zeta<\infty\}})$$
$$=\nn(1-e^{-p\zeta}\mathbbm{1}_{\{\zeta<\infty\}})\PP\left[\int_0^{\overline{X}_\infty} e^{-p\tau_a^+-\gamma\int_0^{\tau_{a}^+}e^{\alpha X_u}du}f(\overline{X}_{\tau_{a}^+})da\right]$$
$$=\nn(1-e^{-p\zeta}\mathbbm{1}_{\{\zeta<\infty\}})\int_0^{\infty} \PP\left[e^{-\gamma\int_0^{\tau_{a}^+}e^{\alpha X_u}du};\tau_{a}^+<\expp\right]f(a)da,$$
where the penultimate equality follows from the fact that $\tau^+$ has at most countably many jumps, which are then not seen by Lebesgue measure, and the last equality uses the absence of positive jumps. Using \eqref{exit-1} this expresses as 
$$\AA\left(1-\frac{p}{\Phi(p)\drift}\right)=\nn(1-e^{-p\zeta}\mathbbm{1}_{\{\zeta<\infty\}})\int_0^{\infty}\frac{\II(\gamma)}{\II(\gamma e^{\alpha a})}f(a)e^{-\Phi(p)a}da.$$
Taking $\gamma=0$ and $f=1$ and plugging back in concludes the proof, since under $\PP$ the law of $\overline{X}_\expp$ is exponential of rate $\Phi(p)$.
\end{proof}

\subsection{Law of $\jump$}\label{subsection:law-of-jump}

\begin{proposition}\label{proposition:law-of-jump} 
Assume $X$ has paths of finite variation. For $h\in \mathcal{B}_{(-\infty,0]}/\mathcal{B}_{[0,\infty]}$, $$\PP[h(X_G-\overline{X}_\expp)]=\frac{1}{\Phi(p)\drift}\left[ph(0)+\int h(z)(1-e^{\Phi(p)z})\nu(dz)\right];$$ in other words, for  $g\in \mathcal{B}_{(0,1]}/\mathcal{B}_{[0,\infty]}$, 
\begin{equation}\label{eq:law-of-jump}
\QQ[g(\jump)]=\frac{1}{\Phi(p)\drift}\left[pg(1)+\int g(e^z)(1-e^{\Phi(p)z})\nu(dz)\right].
\end{equation}
\end{proposition}

\begin{remark}\label{remark:jump}
Of course any $\PP_x$ (resp. $\QQ_y$) may replace $\PP$ (resp. $\QQ$) in the above. In the proof we will see that $\nu/\drift$ is the law of $\xi_0$ under $\nn$ (which is otherwise a known fact \cite{rogers}; we include the (short) argument for completeness).
\end{remark}

\begin{proof}
Let  $f\in \mathcal{B}_\mathbb{R}/\mathcal{B}_{[0,\infty]}$. We have by the compensation formulas for $\epsilon$ and $\Delta X$, for any (arbitrary) $q\in (0,\infty)$, $$q \nn\left[f(\xi_0)\right]\Phi(q)^{-1}=q\PP\left[\sum_{a\in \dd}e^{-q\tau_{a-}^+}f(\epsilon_a(0))\right]=q\PP\left[\sum_{t\in \jj}e^{-q t}\mathbbm{1}_{\{\overline{X}_{t-}=X_{t-}\}}f(\Delta X_t)\right]$$
$$=\PP\left[\int_0^\infty \mathbbm{1}_{\{\overline{X}_{t-}=X_{t-}\}}q  e^{-q t}dt\right]\nu(f)=\PP\left[\int_0^\infty \mathbbm{1}_{\{\overline{X}_{t}=X_{t}\}}q  e^{-q t}dt\right]\nu(f)=\frac{q}{\Phi(q) \drift}\nu(f).$$ 
Let now $h$ be bounded; we compute, similarly as in the proof of Proposition~\ref{proposition:laplace-transforms},
$$\PP[h(X_G-\overline{X}_\expp)]=\PP \left[\sum_{a\in \dd}h(\epsilon_a(0))\mathbbm{1}_{\{\tau_{a-}^+<\expp\leq \tau_a^+\}}\right]+h(0)\PP\left(\overline{X}_\expp=X_\expp\right)$$
$$=\PP\left[\sum_{a\in \dd}e^{-p\tau_{a-}^+}h(\epsilon_a(0))(1-e^{-p\zeta(\epsilon_a)}\mathbbm{1}_{\{\zeta(\epsilon_a)<\infty\}})\right]+h(0)\frac{p}{\Phi(p)\drift}$$
$$=\PP\left[\int_0^{\overline{X}_a}e^{-p\tau_{a-}^+}da\right]\nn \left[h(\xi_0)(1-e^{-p\zeta}\mathbbm{1}_{\{\zeta<\infty\}})\right]+h(0)\frac{p}{\Phi(p)\drift}$$
$$=\Phi(p)^{-1}\lim_{t\downarrow 0}\nn\left[h(\xi_0)(1-e^{-p\zeta}\mathbbm{1}_{\{\zeta<\infty\}});t<\zeta\right]+h(0)\frac{p}{\Phi(p)\drift}$$
$$=\Phi(p)^{-1}\lim_{t\downarrow 0}\nn\left[h(\xi_0)e^{-pt}(1-e^{-p\zeta}\mathbbm{1}_{\{\zeta<\infty\}})\circ \theta_t;t<\zeta\right]+h(0)\frac{p}{\Phi(p)\drift},$$ (by dominated convergence, because $(1-e^{-pt})\mathbbm{1}_{\{t<\zeta\}}\leq 1-e^{-p\zeta}\mathbbm{1}_{\{\zeta<\infty\}}$ and $\nn(1-e^{-p\zeta}\mathbbm{1}_{\{\zeta<\infty\}})<\infty$)
$$=\Phi(p)^{-1}\lim_{t\downarrow 0}\nn\left[h(\xi_0)\PP_{\xi_t}\left(1-e^{-p\tau_0^+}\mathbbm{1}_{\{\tau_0^+<\infty\}})\right);t<\zeta\right]+h(0)\frac{p}{\Phi(p)\drift}$$
(by the Markov property of $\nn$)
$$=\Phi(p)^{-1}\lim_{t\downarrow 0}\nn\left[h(\xi_0)\left(1-e^{\Phi(p)\xi_t}\right);t<\zeta\right]+h(0)\frac{p}{\Phi(p)\drift}=\Phi(p)^{-1}\nn\left[h(\xi_0)\left(1-e^{\Phi(p)\xi_0}\right)\right]+h(0)\frac{p}{\Phi(p)\drift}$$
(by dominated convergence because $\left(1-e^{\Phi(p)\xi_t}\right)\mathbbm{1}_{\{t<\zeta\}}\leq \left(1-e^{\Phi(p)\underline{\xi}_{\zeta}}\right)$ and $\nn\left(1-e^{\Phi(p)\underline{\xi}_{\zeta}}\right)=\int_0^\infty \Phi(p)e^{-\Phi(p)a}\nn\left(-\underline{\xi}_{\zeta}>a\right)da$, which quantity is finite when $X$ has paths of finite variation, as follows from  \eqref{eq:scale-excursion} 
and the fact that in this case $W(0)>0$). By the first part of this proof the claims follow.
\end{proof}

\subsection{Laplace transform of $T_0-L$ given $\overline{Y}_\infty$ and $\jump$}
\begin{definition}\label{definition:NN}
We introduce the function \footnotesize
$$
\NN(y):=\JJ(\beta  y^\alpha)+\frac{\alpha}{\Phi(p)}\left[\frac{\JJ(\beta  y^\alpha)}{\II(\beta y^\alpha)}\sum_{k=1}^\infty k\bb_k(\beta  y^\alpha)^k-\sum_{k=1}^\infty k\aa_k(\beta  y^\alpha)^k\right],\quad y\in (0,\infty),\beta\in [0,\infty),
$$\normalsize
where the expression must be understood in the limiting sense (as $\alpha\to \Phi(p)/m$), when $\Phi(p)=\alpha m$ for some $m\in \mathbb{N}$: 
the limit is seen to exist and identified in Remark~\ref{remark:limit} to follow. 
\end{definition}
\begin{remark} \label{remark:limit}
Suppose $\Phi(p)=m\alpha$ for an $m\in \mathbb{N}$. Then we may write, for $\alpha'\in (\frac{\Phi(p)}{m+1},\frac{\Phi(p)}{m-1})\backslash \{\frac{\Phi(p)}{m}\}$, setting  provisionally $\widetilde{\bbb_k}:=(\prod_{l=1}^k(\psi(m\alpha'+l\alpha')-p))^{-1}$ for $k\in \mathbb{N}_0$,
$$\NNN(y)=\sum_{k=0}^{m-1}\aaa_k (\beta y^{\alpha'})^k+\frac{\alpha'}{\Phi(p)}\left[\frac{\sum_{k=0}^{m-1}\aaa_k (\beta y^{\alpha'})^k}{\II(\beta y^{\alpha'})}\sum_{k=1}^\infty k\bbb_k(\beta  y^{\alpha'})^k-\sum_{k=1}^{m-1}k\aaa_k(\beta  y^{\alpha'})^k\right]$$\footnotesize
$$+\frac{\aaa_{m-1}(\beta y^{\alpha'})^m}{\psi(\alpha' m)-p}\left[\sum_{k=0}^\infty \widetilde{\bbb_k}(\beta y^{\alpha'})^k+\frac{\alpha'}{\Phi(p)}\left(\frac{\sum_{k=0}^\infty \widetilde{\bbb_k}(\beta y^{\alpha'})^k}{\II(\beta y^{\alpha'})}\sum_{k=1}^\infty k\bbb_k(\beta  y^{\alpha'})^k-\sum_{k=0}^\infty \widetilde{\bbb_k}(\beta y^{\alpha'})^k(k+m)\right)\right]$$\normalsize
$$\xrightarrow{\alpha'\to \alpha}\sum_{k=0}^{m-1}\aa_k (\beta y^{\alpha})^k+\frac{\alpha}{\Phi(p)}\left[\frac{\sum_{k=0}^{m-1}\aa_k (\beta y^{\alpha})^k}{\II(\beta y^{\alpha})}\sum_{k=1}^\infty k\bb_k(\beta  y^{\alpha})^k-\sum_{k=1}^{m-1}k\aa_k(\beta  y^{\alpha})^k\right]$$
$$+\frac{\aa_{m-1}(\beta y)^{\alpha m}}{\psi'(\Phi(p))\Phi(p)}\left[\alpha\left(1-\frac{\sum_{k=1}^\infty k \bb_k(\beta y^\alpha)^k}{\II(\beta y^\alpha)}\right)\sum_{k=1}^\infty\cc_k(\beta y^\alpha)^k-\II(\beta y^\alpha)\right]=\NN(y)$$
(recall the $\cc_k$ from Remark~\ref{remark:limits:ii}).
\end{remark}

\begin{proposition}\label{proposition:laplace-of-residual-time}
Let $\beta\in [0,\infty)$.
\leavevmode
\begin{enumerate}[(i)]
\item\label{prop:fv} If $X$ has finite variation, then 
\begin{equation}\label{eq:cond-laplace-jump-sup} 
\PP_x\left[ e^{-\beta\int_G^\expp e^{\alpha X_u}du}\vert \overline{X}_\expp,X_G\right]=\frac{\MM( e^{ X_G}, e^{ \overline{X}_\expp})}{1-e^{\Phi(p)(X_G-\overline{X}_\expp)}}\mathbbm{1}_{(-\infty,0)}(X_G-\overline{X}_\expp)+\mathbbm{1}_{\{0\}}(X_G-\overline{X}_\expp)
\end{equation}
 a.s.-$\PP_x$ for all $x\in \mathbb{R}$; in other words 
\begin{equation}\label{eq:cond-laplace-jump-sup-Y} 
\QQ_y\left[ e^{-\beta(T_0-L)}\vert \overline{Y}_\infty,J\right]=\frac{\MM(\overline{Y}_\infty \jump,\overline{Y}_\infty)}{1-\jump^{\Phi(p)}}\mathbbm{1}_{(0,1)}(\jump)+\mathbbm{1}_{\{1\}}\left(\jump\right)
\end{equation}
a.s.-$\QQ_y$ for all $y\in (0,\infty)$.
\item\label{prop:iv:diffusion} If $X$ has infinite variation, then 
\begin{equation}\label{eq:cond-laplace-sup-diffusion} 
\PP_x\left[ e^{-\beta\int_G^\expp e^{\alpha X_u}du}\vert \overline{X}_\expp\right]=\NN(e^{\overline{X}_\expp})
\end{equation}
 a.s.-$\PP_x$ for all $x\in \mathbb{R}$; in other words 
\begin{equation}\label{eq:cond-laplace-jump-sup-Y} 
\QQ_y\left[ e^{-\beta(T_0-L)}\vert \overline{Y}_\infty\right]=\NN(\overline{Y}_\infty)
\end{equation}
a.s.-$\QQ_y$ for all $y\in (0,\infty)$.
\end{enumerate}
\end{proposition}
\begin{remark}
Recall $\overline{X}_\expp=X_G$, i.e. $\jump=1$, a.s. when $X$ has paths of infinite variation.
\end{remark}
\begin{remark}
In the course of the proof we establish, en passant, that the $\kk$ from \eqref{eq:conditioning}  is equal to $1$.
\end{remark}
\begin{proof}
Again we may work without loss of generality under $\PP$. Let $\{f,g\} \subset \mathcal{B}_{\mathbb{R}}/\mathcal{B}_{[0,\infty)}$ be bounded. We are interested in $$\PP\left[e^{-\beta\int_G^\expp e^{\alpha X_u}du}h(X_G-\overline{X}_\expp)f(\overline{X}_\expp)\right]$$\footnotesize
$$=\PP\left[\sum_{a\in \dd}\exp\left\{-\beta e^{\alpha \overline{X}_{\tau_{a-}^+}}\int_0^{\expp-\tau_{a-}^+} e^{\alpha \epsilon_a(u)}du\right\}h(\epsilon_a(0))f(\overline{X}_{\tau_{a-}^+})\mathbbm{1}_{\{\tau_{a-}^+< \expp\leq \tau_a^+\}}\right]+\PP\left[h(0)f(\overline{X}_\expp);\overline{X}_\expp=X_\expp\right]$$\normalsize
(the second term appearing only if $p>0$)\footnotesize
$$=\PP\left[\sum_{a\in \dd}e^{-p\tau_{a-}^+}\exp\left\{-\beta e^{\alpha \overline{X}_{\tau_{a-}^+}}\int_0^{\expp} e^{\alpha \epsilon_a(u)}du\right\}h(\epsilon_a(0))f(\overline{X}_{\tau_{a-}^+})\mathbbm{1}_{\{\expp\leq \zeta(\epsilon_a)\}}\right]+h(0)\PP[f(\overline{X}_\expp)]\PP(\overline{X}_\expp=X_\expp).$$\normalsize
(by the memoryless property of the exponential distribution and because $X$ is independent of $\expp$)
$$=\int_0^\infty e^{-\Phi(p)a}f(a)\int \nn\left[e^{-\beta e^{\alpha a}\mathrm{I}_u}h(\xi_0);u\leq \zeta\right]\Exp_p(du)da+h(0)\int_0^\infty\Phi(p) e^{-\Phi(p)a}f(a)da \frac{p}{\Phi(p)\drift},$$
where $\Exp_p$ is the exponential law of rate $p$ on $\mathcal{B}_{(0,\infty]}$. 
Thus it remains to determine, for $a\in [0,\infty)$, 
$$\BB:=\int \nn\left[e^{-\beta e^{\alpha a}\mathrm{I}_u}h(\xi_0);u\leq \zeta\right]\Exp_p(du)=\lim_{t\downarrow 0}\int \nn\left[e^{-\beta e^{\alpha a}\mathrm{I}_u}h(\xi_0);t<u\leq \zeta\right]\Exp_p(du)$$
(by monotone convergence)
$$=\lim_{t\downarrow 0}\int_{(t,\infty]} \nn\left[e^{-\beta e^{\alpha a}\mathrm{I}_t}h(\xi_0)\left(e^{-\beta e^{\alpha a}\mathrm{I}_{u-t}}\mathbbm{1}_{\{\zeta\geq u-t\}}\right)\circ\theta_t;t<\zeta\right]\Exp_p(du)$$
$$=\lim_{t\downarrow 0}\int_{(t,\infty]} \nn\left[h(\xi_0)\left(e^{-\beta e^{\alpha a}\mathrm{I}_{u-t}}\mathbbm{1}_{\{\zeta\geq u-t\}}\right)\circ\theta_t;t<\zeta\right]\Exp_p(du)$$
(by dominated convergence, because $\nn(\zeta\geq u)<\infty$ for each $u\in (0,\infty]$ and moreover $\nn[1-e^{-p\zeta}\mathbbm{1}_{\{\zeta<\infty\}}]<\infty$)
$$=\lim_{t\downarrow 0}\int_{(t,\infty]} \nn\left[h(\xi_0)\left(e^{-\beta e^{\alpha a}\mathrm{I}_{u-t}}\mathbbm{1}_{\{\zeta\geq u-t\}}\right)\circ\theta_t;t<\zeta\right]e^{pt}\Exp_p(du)$$
$$=\lim_{t\downarrow 0}\int \nn\left[h(\xi_0)\left(e^{-\beta e^{\alpha a}\mathrm{I}_{v}}\mathbbm{1}_{\{\zeta\geq v\}}\right)\circ\theta_t;t<\zeta\right]\Exp_p(dv)$$
$$=\lim_{t\downarrow 0}\int \nn \left[h(\xi_0)\PP_{\xi_t}\left[e^{-\beta e^{\alpha a}\mathrm{I}_v};\tau_0^+\geq v\right];t<\zeta\right]\Exp_p(dv)$$
$$=\lim_{t\downarrow 0} \nn \left[h(\xi_0)\PP_{\xi_t}\left[e^{-\beta e^{\alpha a}\mathrm{I}_\expp};\tau_0^+\geq \expp\right];t<\zeta\right]=\lim_{t\downarrow 0}\nn \left[h(\xi_0)\MM( e^{a+\xi_t}, e^{ a});t<\zeta\right],$$ where we used \eqref{eq:exit-conditioned} in the last equality. 

\ref{prop:fv}. Suppose now first that $X$ is of finite variation. We know already from the proof of Proposition~\ref{proposition:law-of-jump} that $\nn[1-e^{\Phi(p)\underline{\xi}_{\zeta}}]<\infty$. Since, for $t\in (0,\infty)$, $\MM(e^{a+\xi_t},e^{ a})\leq 1-e^{\Phi(p)\xi_t}\leq 1-e^{\Phi(p)\underline{\xi}_{\zeta}}$ on $\{t<\zeta\}$, it follows therefore by dominated convergence that 
$$\BB=\nn\left[h(\xi_0)\MM( e^{a+\xi_0}, e^{ a})\right].$$
Hence, using the fact that ${\overline{X}_\expp}_\star\PP=\Exp_{\Phi(p)}$, Proposition~\ref{proposition:law-of-jump}, Remark~\ref{remark:jump}, and the independence of $\overline{X}_\expp$ from $X_G-\overline{X}_\expp$, it follows that  $$\PP\left[e^{-\beta\int_G^\infty e^{\alpha X_u}du}h(X_G-\overline{X}_\expp)f(\overline{X}_\infty)\right]$$
$$=\PP\left[f(\overline{X}_\infty)h(X_G-\overline{X}_\infty)\left(\frac{\MM( e^{ X_G}, e^{ \overline{X}_\expp})}{1-e^{\Phi(p)(X_G-\overline{X}_\expp)}}\mathbbm{1}_{(-\infty,0)}(X_G-\overline{X}_\expp)+\mathbbm{1}_{\{0\}}(X_G-\overline{X}_\expp)\right)\right]$$ and \eqref{eq:cond-laplace-jump-sup} is proved.

\ref{prop:iv:diffusion}. Now let $X$ be of infinite variation; take $h=1$. Because the coordinate projection $(\DD\ni \xi\mapsto \xi_t)$ is continuous in the Skorokhod topology at all paths for which $t\in (0,\infty)$ is a continuity point, and in particular (by the Markov property of $\nn$ and since $X$ has no fixed points of discontinuity a.s.) $\nn$-a.e., it follows from \eqref{eq:conditioning} that 
$$\BB=\kk\lim_{t\downarrow 0}\lim_{x\uparrow 0} \frac{\PP_x\left[\MM( e^{a+X_t}, e^{ a});t<\tau_0^+\right]}{\gg(x)}.$$
By the preceding the expression inside the limit $\lim_{t\downarrow 0}$, call it $r(t)$, is bounded by $\nn(1-e^{-p\zeta}\mathbbm{1}_{\{\zeta<\infty\}})e^{pt}$ and the limit $\lim_{t\downarrow 0}r(t)$ exists a priori. Therefore $\lim_{q\to\infty}\int_0^\infty e^{-q t}q r(t)dt=\lim_{q\to\infty}\frac{q}{q-p}\int_0^\infty e^{-(q-p) t}(q-p) e^{-pt}r(t)dt=r(0+)$, and since $\lim_{x\downarrow 0}\frac{\gg(x)}{-x}=1$, we obtain
$$\BB=\kk \lim_{q \to\infty}\int_0^\infty q e^{-q t}\lim_{x\uparrow 0} \frac{\PP_x\left[\MM( e^{a+X_t}, e^{ a});t<\tau_0^+\right]}{-x}dt.$$ 
Besides, for all $t\in (0,\infty)$, $\MM( e^{a+X_t}, e^{ a})\leq 1-e^{\Phi(p)X_t}$ on $\{t<\tau_0^+\}$, and since $\PP_x\left[1-e^{\Phi(p)X_t};t<\tau_0^+\right]=\PP_x[(1-e^{-p(\tau_0^+-t)}\mathbbm{1}_{\{\tau_0^+<\infty\}});t<\tau_0^+]\leq \PP_x(\expp\leq \tau_0^+)=1-e^{\Phi(p)x}$, dominated convergence yields  $$\BB=\kk \lim_{q \to\infty}\lim_{x\uparrow 0} \frac{\int_0^\infty q e^{-q t}\PP_x\left[\MM( e^{a+X_t}, e^{ a});t<\tau_0^+\right]dt}{-x}$$
\begin{equation}\label{eq:double-limit}
=\kk \lim_{q \to\infty}q \lim_{x\downarrow  0}\int_{0}^\infty\MM( e^{a-y}, e^{ a})\frac{e^{-\Phi(q)x}\Wm(y)-\Wm(y-x)}{x}dy,
\end{equation}
by \eqref{eq:resolvent}.

Next, as the integral of a resolvent density,  for all $q\in(0,\infty)$ and all $x\in (0,\infty)$, one has that $\int_{0}^\infty e^{-\Phi(q)x}\Wm(y)-\Wm(y-x)dy$ is finite. At the same time we know from \eqref{eq:laplace} that 
\begin{equation}\label{eq:ad-analytic-cont}
(\psi(\lambda)-q)\int_0^\infty e^{-\lambda y}\left(e^{-\Phi(q)x}\Wm(y)-\Wm(y-x)\right)dy=e^{-\Phi(q)x}-e^{-\lambda x}
\end{equation}
at least for $\lambda\in (\Phi(q),\infty)$. But by the theorems of Cauchy, Morera and Fubini, the left-hand side, and clearly the right-hand side are analytic/can be extended to analytic functions in $\lambda\in \{z\in \mathbb{C}:\Re z>0\}$. Hence, by the principle of permanence for analytic function, the equality \eqref{eq:ad-analytic-cont} prevails for $\lambda\in (0,\infty)$. 
Taking limits, by monotone convergence/continuity, we conclude that 

$$
\int_0^\infty e^{-\lambda y}\left(e^{-\Phi(q)x}\Wm(y)-\Wm(y-x)\right)dy=\frac{e^{-\Phi(q)x}-e^{-\lambda x}}{\psi(\lambda)-q}
$$ 
for all $\lambda\in [0,\infty)$, provided the right-hand side is interpreted in the limiting sense at $\lambda=\Phi(q)$.

Consequently, integrating term-by-term (via linearity and monotone or dominated convergence) in \eqref{eq:double-limit}, we obtain, assuming $\Phi(p)\notin\alpha\mathbb{N}$,

$$\BB=\kk \lim_{q \to\infty}^{*}q \lim_{x\downarrow  0}\frac{\sum_{k=0}^\infty \aa_k (\beta e^{\alpha a})^k\frac{e^{-\Phi(q)x}-e^{-\alpha k x}}{\psi(\alpha k)-q}-\frac{\JJ(\beta e^{\alpha a})}{\II(\beta e^{\alpha a})}\sum_{k=0}^\infty \bb_k(\beta e^{\alpha a}) ^k\frac{e^{-\Phi(q)x}-e^{-(\Phi(p)+k\alpha) x}}{\psi(\Phi(p)+k\alpha)-q}}{x}$$ where $\lim^\star$ indicates that we take the limit along a sequence $(q_k)_{k\in \mathbb{N}_0}$ that uniformly (we use this /for convenience/ later on when arguing dominated convergence) avoids the grid $\psi(\alpha\mathbb{N}_0\cup (\Phi(p)+\alpha\mathbb{N}_0))$. 

Then by dominated convergence (recall the elementary estimate $1-e^{-u}\leq u$, $u\in [0,\infty)$)
$$\BB=\kk \lim_{q \to\infty}^{*}q \left[\sum_{k=0}^\infty \aa_k (\beta e^{\alpha a})^k\frac{\alpha k-\Phi(q)}{\psi(\alpha k)-q}-\frac{\JJ(\beta e^{\alpha a})}{\II(\beta e^{\alpha a})}\sum_{k=0}^\infty \bb_k(\beta e^{\alpha a}) ^k\frac{\Phi(p)+k\alpha-\Phi(q)}{\psi(\Phi(p)+k\alpha)-q}\right]$$
\footnotesize
$$=\kk \lim_{q \to\infty}^*q\left[\sum_{k=0}^\infty \aa_k(\beta  e^{\alpha a})^k\left(\frac{\alpha k-\Phi(q)}{\psi(\alpha k)-q}-\frac{\Phi(q)}{q}\right)-\frac{\JJ(\beta  e^{\alpha a})}{\II(\beta  e^{\alpha a})}\sum_{k=0}^\infty \bb_k(\beta  e^{\alpha a})^k\left(\frac{\Phi(p)+\alpha k-\Phi(q)}{\psi(\Phi(p)+k\alpha)-q}-\frac{\Phi(q)}{q}\right)\right].$$\normalsize

$$=\kk \lim_{q \to\infty}^{*}\left[\sum_{k=0}^\infty \aa_k(\beta  e^{\alpha a})^k\frac{q\alpha k}{\psi(\alpha k)-q}-\frac{\JJ(\beta  e^{\alpha a})}{\II(\beta  e^{\alpha a})}\sum_{k=0}^\infty \bb_k(\beta  e^{\alpha a})^k\frac{q(\Phi(p)+\alpha k)}{\psi(\Phi(p)+k\alpha)-q}\right]$$
$$+\kk \lim_{q \to\infty}^{*}\sum_{k=0}^\infty \aa_k(\beta  e^{\alpha a})^k\left(\frac{(\alpha k-\Phi(q))\psi(\alpha k)}{\psi(\alpha k)-q}-\frac{\alpha k\psi(\alpha k)}{\psi(\alpha k)-q}\right)$$
$$-\kk\frac{\JJ(\beta  e^{\alpha a})}{\II(\beta  e^{\alpha a})} \lim_{q \to\infty}^{*}\sum_{k=0}^\infty \bb_k(\beta  e^{\alpha a})^k\left(\frac{(\alpha k+\Phi(p)-\Phi(q))\psi(\alpha k+\Phi(p))}{\psi(\alpha k+\Phi(p))-q}-\frac{(\alpha k+\Phi(p))\psi(\alpha k+\Phi(p))}{\psi(\alpha k+\Phi(p))-q}\right)$$
$$=\kk \lim_{q \to\infty}^{*}\left[\sum_{k=0}^\infty \aa_k(\beta  e^{\alpha a})^k\frac{q\alpha k}{\psi(\alpha k)-q}-\frac{\JJ(\beta  e^{\alpha a})}{\II(\beta  e^{\alpha a})}\sum_{k=0}^\infty \bb_k(\beta  e^{\alpha a})^k\frac{q(\Phi(p)+\alpha k)}{\psi(\Phi(p)+k\alpha)-q}\right]$$
(by dominated convergence, noting that $\frac{z-w}{\psi(z)-\psi(w)}$ is bounded in $z\ne w$, $\{z,w\}\subset [\Phi(p),\infty)$ by the strict convexity of $\psi$, and that $\lim_{q\to\infty}\frac{\Phi(q)}{q}=0$ because $\lim_\infty\psi'= \infty$)
$$=\kk \lim_{q \to\infty}^{*}\sum_{k=0}^\infty \aa_k(\beta  e^{\alpha a})^k\left(\frac{q\alpha k}{\psi(\alpha k)-q}+\alpha k-\alpha k\right)$$
$$-\kk \frac{\JJ(\beta  e^{\alpha a})}{\II(\beta  e^{\alpha a})}\lim_{q \to\infty}^{*}\sum_{k=0}^\infty \bb_k(\beta  e^{\alpha a})^k\left(\frac{q(\Phi(p)+\alpha k)}{\psi(\Phi(p)+k\alpha)-q}+\Phi(p)+\alpha k-(\Phi(p)+\alpha k)\right)$$
$$=\kk\left[\frac{\JJ(\beta  e^{\alpha a})}{\II(\beta  e^{\alpha a})}\sum_{k=0}^\infty \bb_k(\beta  e^{\alpha a})^k(\Phi(p)+\alpha k)-\sum_{k=1}^\infty \aa_k(\beta  e^{\alpha a})^k\alpha k\right]$$
(again by dominated convergence).
Plugging in $f=1$, $\beta=0$, we identify $\kk=1$. The case $\Phi(p)\in \alpha\mathbb{N}$ follows by taking limits.
\end{proof}

\subsection{Conditional temporal splitting at the maximum}
Combining our results we arrive at

\begin{theorem}\label{theorem}
Let $\beta\in [0,\infty)$.
\begin{enumerate}
\item Let $X$ be of finite variation. Then the random variables $L$ and $T_0-L$ are independent given $\overline{Y}_\infty$ and $\jump$ (and also just given $\overline{Y}_\infty$), which in turn are independent. The law of $\jump$ is given by \eqref{eq:law-of-jump}, and one has the conditional factorization 

$$\QQ_y[e^{-\beta T_0}\vert \overline{Y}_\infty,\jump]=\frac{\II(\beta y^\alpha)}{\II(\beta {\overline{Y}_\infty}^\alpha)}\times \left[\frac{\MM(\overline{Y}_\infty \jump,\overline{Y}_\infty)}{1-\jump^{\Phi(p)}}\mathbbm{1}_{(0,1)}(\jump)+\mathbbm{1}_{\{1\}}\left(\jump\right)\right],$$ a.s.-$\QQ_y$ for all $y\in \mathbb{R}$, where the two factors either side of $\times$ correspond to the conditional expectations $\QQ_y[e^{-\beta L}\vert \overline{Y}_\infty]\overset{\text{a.s.-$\QQ_y$}}{=}\QQ_y[e^{-\beta L}\vert \overline{Y}_\infty,\jump]$ and $\QQ_y[e^{-\beta (T_0-L)}\vert \overline{Y}_\infty,\jump]$, respectively.
\item Let $X$ be of infinite variation. Then $\jump=1$ a.s., the random variables $L$ and $T_0-L$ are independent given $\overline{Y}_\infty$, and one has the conditional factorization 

$$\QQ_y[e^{-\beta T_0}\vert \overline{Y}_\infty,\jump]=\frac{\II(\beta y^\alpha)}{\II(\beta {\overline{Y}_\infty}^\alpha)}\times\NN(\overline{Y}_\infty),$$ a.s.-$\QQ_y$ for all $y\in \mathbb{R}$, where the two factors either side of $\times$ correspond to the conditional expectations $\QQ_y[e^{-\beta L}\vert \overline{Y}_\infty]$ and $\QQ_y[e^{-\beta (T_0-L)}\vert \overline{Y}_\infty]$, respectively.
\end{enumerate}
In either case the law of $\overline{Y}_\infty$  is exponential of rate $\Phi(p)$.
\end{theorem}
\begin{proof}
This follows from Propositions~\ref{proposition:law-of-jump},~\ref{proposition:laplace-transforms} and~\ref{proposition:laplace-of-residual-time}, and from the comments concerning the (conditional) independences in $(L,\overline{Y}_\infty,\jump,T_0-L)$  made in the Introduction (as a consequence of the independence statement of the Wiener-Hopf factorization for $X$). 
\end{proof}

\section{Concluding remarks/applications}\label{section:concluding}

We conclude with some indications of applications and possible further avenues of research (besides Question~\ref{section:concluding} that we already pointed out in the Introduction).

\subsection{Expected discounted payoff of a ``regret'' lookback option}
One immediate application of the above that springs to mind is to the computation of the expected discounted  payoff (under the ``physical'' measure) of a lookback option on the stock of a company that one sees eventually going bankrupt and whose price is modeled by the process $Y$. The idea being that one holds equity in the company until termination, say for dividends, but at the same time wants an option to provide some hedge against not selling the stock sooner (or indeed at its maximum). It is a ``buy-and-hold'' strategy in the face of the recognition that eventually the company will terminate.  

Specifically, we may imagine that the stock price is given by the process $Y$ under the ``physical'' measure $\QQ_y$ for an initial price $y\in (0,\infty)$. Assumption~\ref{assumption} then means that eventually the  price will hit zero a.s., either abruptly, say as a result of some one-off adverse event, when $p>0$, or continuously, say as a result of gradually deteriorating business conditions, when $p=0$. At the same time we have an option written on the stock that will pay some (nondecreasing) function $f:(0,\infty)\to [0,\infty)$ of the overall maximum $\overline{Y}_\infty$ at termination $T_0$, providing thus some measure of compensation for the ``regret'' of not having sold the stock optimally. 

Assuming a constant risk-free force of interest $r\in [0,\infty)$, then the expected discounted payoff of such an option is simply $$\QQ_y[e^{-rT_0}f(\overline{Y}_\infty)]=\QQ_y[f(\overline{Y}_\infty)\QQ_y[e^{-r L}\vert \overline{Y}_\infty]\QQ_y[e^{-r(T_0-L)}\vert \overline{Y}_\infty,J]],$$ which may easily be expressed using the results of Theorem~\ref{theorem} in terms of (at most) a two-dimensional integral. Such an expectation provides some information on the ``value'' of the option for the investor, though of course it does not correspond to a risk-neutral valuation theoreof.

\subsection{Properties of the joint law of $(L,\overline{Y}_\infty,J,T_0-L)$}
It is immediate from the definition of $\II$ and from Proposition~\ref{proposition:laplace-transforms} (but not obvious a priori) that, for a given $\beta\in[0,\infty)$, $y\in (0,\infty)$, the conditional Laplace transform $\QQ_y[e^{-\beta L}\vert \overline{Y}_\infty]\overset{\text{a.s.}}{=}\frac{\II(\beta Y^\alpha)}{\II(\beta{\overline{Y}_\infty}^\alpha)}$ is decreasing as a function of $\overline{Y}_\infty$. On the other hand it is clear from Definition~\ref{definition:NN} that $\NN(y)$ depends on $\beta$ and $y$ only through $\beta y^\alpha$. Hence it follows from Proposition~\ref{proposition:laplace-of-residual-time}\ref{prop:iv:diffusion} that $\QQ_y[e^{-\beta (T_0-L)}\vert \overline{Y}_\infty]\overset{\text{a.s.}}{=}\NN(\beta {\overline{Y}_\infty}^\alpha)$ is decreasing as a function of $\overline{Y}_\infty$ when $X$ has paths of infinite variation.  In the opposite case, it follows similarly from Definition~\ref{definition:MM}, Corollary~\ref{corollary} and from Proposition~\ref{proposition:laplace-of-residual-time}\ref{prop:fv}, that $\QQ_y[e^{-\beta (T_0-L)}\vert \overline{Y}_\infty,J]\overset{\text{a.s.}}{=}\frac{\MM(\overline{Y}_\infty \jump,\overline{Y}_\infty)}{1-\jump^{\Phi(p)}}\mathbbm{1}_{(0,1)}(\jump)+\mathbbm{1}_{\{1\}}\left(\jump\right)$ is also decreasing in $\overline{Y}_\infty$; however its dependence on $\jump$ is non-trivial, see Figure~\ref{figure}. 

\begin{figure}[h!]
\begin{center}
 \includegraphics{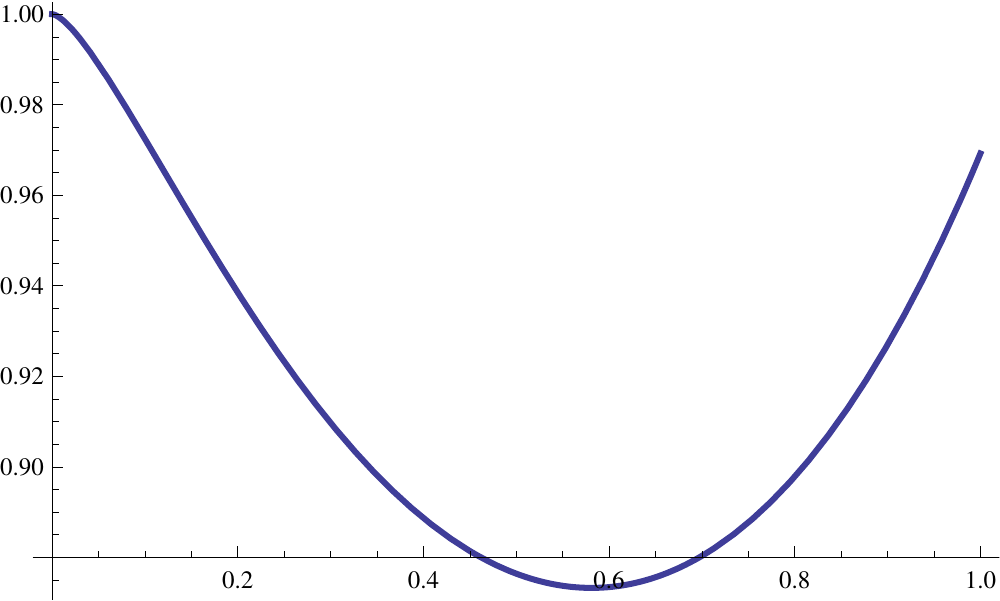}
\caption{The function $(0,1)\ni j\mapsto \frac{\MM(j,1)}{1-j^{\Phi(p)}}$ for $\alpha=2$, $p=1$ and $\psi(\lambda)=\lambda-\frac{\lambda}{\lambda+1}$, $\lambda\in [0,\infty)$, corresponding to $X$ being the difference of a unit drift and of a homogeneous Poisson process of unit intensity.}\label{figure}
\end{center}
\end{figure}
More generally one would be interested in
\begin{question}
What properties of the joint law of $(L,\overline{Y}_\infty,J,T_0-L)$ can be deduced based on the results of Theorem~\ref{theorem} (or otherwise)?
\end{question}
We have given a flavor of this in the above, but do not pursue this problem any further here.




\bibliographystyle{plain}
\bibliography{pssMp}
%

\end{document}